\documentclass{amsart}
\usepackage{amsfonts,amscd,amssymb,amsmath,amsthm}
\usepackage{graphicx,mathrsfs,appendix}
\usepackage{xparse,centernot,caption,subcaption}
\usepackage[T1]{fontenc}
\usepackage{ctable,color}
\begin{document}

\newtheorem{theorem}{Theorem}[section]
\newtheorem{lemma}[theorem]{Lemma}
\newtheorem{corollary}[theorem]{Corollary}
\newtheorem{conjecture}[theorem]{Conjecture}
\newtheorem{cor}[theorem]{Corollary}
\newtheorem{proposition}[theorem]{Proposition}
\theoremstyle{definition}
\newtheorem{definition}[theorem]{Definition}
\newtheorem{example}[theorem]{Example}
\newtheorem{claim}[theorem]{Claim}
\newtheorem{remark}[theorem]{Remark}

\newenvironment{pfofthm}[1]
{\par\vskip2\parsep\noindent{\sc Proof of\ #1. }}{{\hfill
$\Box$}
\par\vskip2\parsep}
\newenvironment{pfoflem}[1]
{\par\vskip2\parsep\noindent{\sc Proof of Lemma\ #1. }}{{\hfill
$\Box$}
\par\vskip2\parsep}


\newcommand{\R}{\mathbb{R}}
\newcommand{\T}{\mathcal{T}}
\newcommand{\C}{\mathcal{C}}
\newcommand{\G}{\mathcal{G}}
\newcommand{\Z}{\mathbb{Z}}
\newcommand{\Q}{\mathbb{Q}}
\newcommand{\E}{\mathbb E}
\newcommand{\N}{\mathbb N}

\newcommand{\barray}{\begin{eqnarray*}}
\newcommand{\earray}{\end{eqnarray*}}

\newcommand{\beq}{\begin{equation}}
\newcommand{\eeq}{\end{equation}}


\renewcommand{\Pr}{\mathbb{P}}
\newcommand{\as}{\text{a.s.}}
\newcommand{\Prob}{\Pr}
\newcommand{\Exp}{\mathbb{E}}
\newcommand{\expect}{\Exp}
\newcommand{\1}{\mathbf{1}}
\newcommand{\prob}{\Pr}
\newcommand{\pr}{\Pr}
\newcommand{\filt}{\mathcal{F}}
\DeclareDocumentCommand \one { o }
{%
\IfNoValueTF {#1}
{\mathbf{1}  }
{\mathbf{1}\left\{ {#1} \right\} }%
}
\newcommand{\Bernoulli}{\operatorname{Bernoulli}}
\newcommand{\Binomial}{\operatorname{Binom}}
\newcommand{\Beta}{\operatorname{Beta}}
\newcommand{\Binom}{\Binomial}
\newcommand{\Poisson}{\operatorname{Poisson}}
\newcommand{\Exponential}{\operatorname{Exp}}


\newcommand{\link}{\mbox{lk}}
\newcommand{\Deg}{\operatorname{deg}}
\newcommand{\vertexsetof}[1]{V\left({#1}\right)}
\renewcommand{\deg}{\Deg}
\newcommand{\oneE}[2]{\mathbf{1}_{#1 \leftrightarrow #2}}
\newcommand{\ebetween}[2]{{#1} \leftrightarrow {#2}}
\newcommand{\noebetween}[2]{{#1} \centernot{\leftrightarrow} {#2}}
\newcommand{\Gap}{\ensuremath{\tilde \lambda_2 \vee |\tilde \lambda_n|}}
\newcommand{\dset}[2]{\ensuremath{ e({#1},{#2})}}
\newcommand{\EL}{{ L}}
\newcommand{\ER}{{Erd\H{o}s--R\'{e}nyi }}
\newcommand{\zuk}{{\.{Z}uk}}


\newcommand{\frm}{\ensuremath{ 2\log\log m}}
\DeclareDocumentCommand \fuzz { o o }
{
\IfNoValueTF {#1}
  { \aleph_M }
  { \IfNoValueTF { #2 }
    { \aleph_{M}^{{#1}} }
    { \aleph_{M}^{{#1}}({#2})}
  }
}
\newcommand{\csubzero}{c_{0}}
\newcommand{\csubone}{c_{1}}
\newcommand{\csubtwo}{c_{2}}
\newcommand{\csubthree}{c_{3}}
\newcommand{\csubstar}{c_{*}}
\newcommand{\INE}{I^{\epsilon}}
\newcommand{\rsp}{1-C\exp(-md^{1/4}\log n)}
\newcommand{\lc}{\ensuremath{ \operatorname{light}(x,y)}}
\newcommand{\hc}{\ensuremath{ \operatorname{heavy}(x,y)}}
\DeclareDocumentCommand \pam { O{m} }
{
P_{{#1}}
}
\DeclareDocumentCommand \sca { O{j} }
{ Q_{ {#1} } }
\DeclareDocumentCommand \mga { O{j} }
{ A_{ {#1} } }
\DeclareDocumentCommand \mgb { O{j} O{c} }
{ Q_{ {#1} }^{ {#2} } }
\DeclareDocumentCommand \mgc { O{j} O{c} }
{ U_{ {#1} }^{ {#2} } }
\DeclareDocumentCommand \Filt { O{j} }
{ \mathscr{F}_{{#1}} }

\title{High degree vertices in the Power of Choice\\
model combined with Preferential Attachment}
\author{Yury Malyshkin}
\address{Tver State University
}
\email{yury.malyshkin@mail.ru}
\date{\today}
\maketitle

\begin{abstract}
We find assimpotics for the first $k$ highest degrees of the degree distribution in an evolving tree model combining the local choice and the preferential attachment.  
In the considered model, the random graph is constructd in the following way. At each step, a new vertex is introduced. 
Then, we connect it with one (the vertex with the largest degree is chosen) of $d$ ($d>2$) possible neighbors, which are sampled from the set of the existing vertices with the probability proportional to their degrees. It is known that the maximum of the degree distribution in this model has linear behavior. We prove that $k$-th highest dergee has a sublinear behavior with a power depends on $d$.  This contrasts sharply with what is seen in the preferential attachment model without choice, where all highest degrees in the degree distribution has the same sublinear order.  The proof is based on showing that the considered tree has a persistent hub by comparison with the standard preferential attachment model, along with martingale arguments.
\end{abstract}

\section{Introduction}
In the present work, we further explore how the addition of choice (see, e.g., \cite{DSKrM,KR13,MP14,MP15}) affects the standart preferential attachment model (see~\cite{barabasi,KrReLe}). The preferential attachment graph model is a time-indexed inductively constructed sequence of graphs, formed in the following way. We start with some initial graph and then on each step we add a new vertex and an edge between it and one of the old vertices, chosen with probability proportional to its degree. Many different properties of this model have been obtained in both the math and physics literature (see~\cite{barabasi, KrReLe,Mori, DvdHH}). 

In the current work, we are interested in the first $k$ maximums of the degree distribution. For the preferential attachment model, this problem is studied in~\cite{FFF04}.  It is shown in \cite{FFF04} that the $k$ highest degrees $\Delta_i(n),$ $i\in\{1,...,k\}$, at time $n$ satisfy
$$\frac{n^{1/2}}{g(n)}\leq\Delta_1(n)\leq g(n)n^{1/2}\, \text{and}\,  \frac{n^{1/2}}{g(n)}\leq\Delta_i(n)\leq \Delta_{i-1}(n)-\frac{n^{1/2}}{g(n)},\quad k\geq i\geq 2,$$
with high probability for any function $g(n)$ with $g(n)\rightarrow\infty$ as $n\rightarrow\infty$.
In~\cite{MP15}, the limited choice is introduced into the preferential attachment model. More specifically, at each step we independently (from each other) choose $d$ existing vertices with a probability proportional to their degrees and connect the new vertex with the vertex with the smallest degree. In~\cite{MP15}, it is shown that the maximal degree at time $n$ in such a model grows as $\log\log n/\log d$ with high probability. If instead of a vertex with the smallest degree we pick one with the highest degree, we would get the max-choice model that was introduced in \cite{MP14}. In~\cite{MP14}, the exact first-order asymptotics for the maximal degree in this model was obtained and almost sure convergence of the appropriately scaled maximal degree was shown.  In the current work, we provide such asymptotics for $k$ highest degrees.

Let us describe the max-choice model. Fix $d\in\mathbb{N}$, $d\geq 2$. Introduce a countable non-random set of vertices $V=\{v_i,\;i\in\mathbb{N}\}$. Define a sequence of random trees $\{ P_n \}$, $n\in\mathbb{N}$, by the following inductive rule. Let $P_1$ be the one-edge tree which consists of vertices $v_1$ and $v_2$ and an edge between them. Given $P_{n}$, we construct $P_{n+1}$ by adding one vertex and drawing one edge in the following way. 

First, we add a vertex $v_{n+2}$ to $P_n$, hence the vertices set $V(P_{n+1})$ of $P_{n+1}$ is $V(P_{n+1})=\{v_i,\;i=1,...,n+2\}$. Note that the randomness of $P_n$ caused by its edge set $\mathcal{E}_n$. Denote $\mathcal{F}_n=\sigma\{\mathcal{E}_1,...,\mathcal{E}_n\}$. Let $X^1_n,\ldots,X^d_n$ be i.i.d. vertices of $V( P_n )$ chosen with the conditional probability
\[
\Pr \left[
X^1_n = v_i|\mathcal{F}_{n}
\right] = \frac{\deg v_i(n)}{2n},\quad v_i\in V(P_n),
\]
where $\deg v_i(n)$ is the degree of $v_i$ in $P_n$ (note that, $\sum_{v_i}\deg v_i(n)=2n$). 

Second, create a new edge between $v_{n+2}$ and $Y_n,$ where $Y_n$ is whichever of $X^1_n$,...,$X^d_n$ has the largest degree. In the case of a tie, choose according to an independent fair coin toss (this choice will not affect the degree distribution). This model is called the \emph{max-choice preferential attachment tree model}.
For any fixed $k\in\mathbb{N}$, let $M_{1}(n)\geq M_2(n)\geq...\geq M_k(n)$ be the degrees of $k$ highest degree vertices at time $n$ (if there are less then $k$ vertices at time $n$ put $M_k(n)=1$). 

Let us formulate our main theorem:
\begin{theorem}
\label{thm:max_degree}

For $d\in\mathbb{N}$, $d>2$, $k\in\mathbb{N}$, $k>1$ and any $\epsilon>0$,
\[
\lim_{n\to\infty} \mathbb{P}(n^{c^{d-1}d/2-\epsilon}<M_k(n)<n^{c^{d-1}d/2+\epsilon}) = 1,
\]
where $c=1-x^{\ast}/2$, $x^{\ast}$ is the unique positive solution of the equation $1-(1-x/2)^{d}=x$ in the interval $0\leq x\leq 1$.
\end{theorem}

Our proof is based on the existence of the $k$-th \emph{persistent hub}, i.e. a single vertex that in some finite random time becomes the $k$-th highest degree vertex for all time after. Using this, instead of analyzing the $k$-th highest degree over all vertices we effectively only need to analyze the degree of just one vertex. The existence of the $k$-th persistent hub is stated in the following result. 

\begin{proposition}
\label{prop:persistent_hub}
There exist random variables $N_l$ and $K_l$, $1\leq l\leq k$, that are finite almost surely so that at any time $n \geq N_l$, $\deg v_{K_l}(n)=M_l(n)$ and $M_1(n)>M_2(n)>...>M_l(n)>\deg v_{i}(n)$ for any $i\neq K_1,...,K_l$.
\end{proposition}

The purpose of this proposition is to simplify analisys of the dynamics of $M_k(n)$.
Indeed, let $L_k(n)$ be the number of vertices at time $n$ that has degree equal to $M_k(n)$. The effect of Proposition~\ref{prop:persistent_hub} is that for some random and sufficiently large $N_k < \infty$, $L_k(n) = 1$ for all $n \geq N_k.$

If $M_{k-1}(n)=M_k(n)$, then $M_k(n+1)=M_{k}(n)$, cause $M_{k-1}(n)$ and $M_k(n)$ could not be increased at the same time and we should increase $M_{k-1}(n)$ before $M_k(n)$. If $M_{k-1}(n)>M_k(n)$, to increase $M_k(n)$ we need to draw an edge to a vertex with the degree $M_k(n)$. Therefore the dynamics of $M_k(n)$ is given by the formula
$$M_k(n+1)-M_k(n) = \1\{\deg Y_{n+1}(n) =M_{k}(n),\,M_{k-1}(n)>M_{k}(n)\},\;\text{with}$$
$$\mathbb{E}(M_k(n+1)-M_k(n)|\mathcal{F}_n)$$
$$=\left(\hat{c}^d_k(n) - \left(\hat{c}_k(n)-\frac{M_k(n)L_k(n)}{2n}\right)^{d}\right) \1\{M_{k-1}(n)>M_{k}(n)\},$$
where
$$\hat{c}_l(n)=1-\frac{1}{2n}\sum_{i=1}^{l-1}M_i(n),\;1\leq l\leq k,\;n\in\mathbb{N}.$$
Note that $\hat{c}_l(n)\geq 0$ cause the sum of the degrees is $2n$.
From here, we will reffer to $\hat{c}^d_k(n) - \left(\hat{c}_k(n)-\frac{M_k(n)L_k(n)}{2n}\right)^{d}$ as $p_{n,k}$. Note that cause $M_k(n+1)-M_k(n)$ could only take values $0$ and $1$, if $M_{k-1}(n)>M_{k}(n)$ then $p_{n,k}$ equals to the probability to increase $k$-th maximal degree at the $n$-{th} step conditional on $\mathcal{F}_{n}$.

Before starting the proof, let us describe its structure and main ideas.
We will prove Proposition~\ref{prop:persistent_hub} and Theorem~\ref{thm:max_degree} using an induction over $k$. To do so, we consider them as independend theorems for each $k$. For $k=1$, the convergence $\frac{M_1(n)}{n}\to x^{\ast}$ almost surely and the existence of the persistent hub were proven in \cite{MP14}. We will fix $k_0>1$ and, using statements of Theorem~\ref{thm:max_degree} and Proposition~\ref{prop:persistent_hub} for $k<k_0$ (from here we reffer to them as induction hypothesis), prove them for $k=k_0$. In Section 2, we prove initial estimates using Theorem~\ref{thm:max_degree} for $k<k_0$. In Section 3, we use these estimates to prove the existence of the persistent hub and, so, prove  Proposition~\ref{prop:persistent_hub} for $k=k_0$. In Section 4, we use Proposition~\ref{prop:persistent_hub} along with lemmas from Section 2 to prove Theorem~\ref{thm:max_degree} for $k=k_0$.

\section{Initial estimates}
\label{sec:apriori}
We assume that Theorem~\ref{thm:max_degree} and Proposition~\ref{prop:persistent_hub} hold for $k<k_0$. In this section, we obtain an initial estimate on $M_{k_0}(n)$ along with some technical lemmas.

Recall that $c=1-x^{\ast}/2$, where $x_{\ast}$ is the solution of the equation $1-(1-x/2)^{d}=x$ in the interval $0\leq x\leq 1$. Define the function
\[
f(x,y)=\frac 1 2 \sum_{i=0}^{d-1}y^{d-i-1}(y-x/2)^{i},\;x,y\in\mathbb{R}_+.
\]
Note that $f(x,y)=\frac{y^d-(y-x/2)^d}{x}$ for $x\neq 0$.
We will need the following estimates.
\begin{lemma}
\label{lem:f_est}
$f(x,y)<1$ for $y^{d-1}<2/d$ and $0\leq x\leq 2y$.
\end{lemma}
\begin{proof}
$f(0,y)=d/2y^{d-1}<1.$ Since $f(x,y)$ is the decreasing function over $x$ for $0\leq x\leq 2y$, we have that $f(x,y)<1$ for corresponding $x$.
\end{proof}
\begin{lemma}
\label{lem:c_est}
$c^{d-1}<2/d$.
\end{lemma}
\begin{proof}
Note that $x^{\ast}=1-(1-x^{\ast}/2)^d=1-c^d$. Therefore,
$$c^{d-1}=\frac{c^d}{c}=\frac{1-x^{\ast}}{1-x^{\ast}/2}<2(1-x^{\ast}).$$
Now show that $x^{\ast}>1-1/d$. Due to convexity of $1-(1-x/2)^{d}$ on $[0,2]$, it is enough to show that $1-(1-(1-1/d)/2)^d-(1-1/d)>0$:
$$1-(1-(1-1/d)/2)^d-(1-1/d)=1/d-\left(\frac{d+1}{2d}\right)^d=1/d-(1/2)^d(1+1/d)^d>0,$$
for $d>2$ (could be easily proved by an induction starting with $d=3$).
\end{proof}

We will frequently use the following lemma of~\cite{Galashin}.
\begin{lemma}
\label{lem:numbers}
Suppose that a sequence of positive numbers $r_n$ satisfies
\[
r_{n+1} = r_n\left(1+\frac{\alpha}{n+x}\right),~n \geq k
\]
for fixed $\alpha > 0,$ $k > 0$ and $x.$  Then $r_n/n^{\alpha}$ has a positive limit.
\end{lemma}
Now, we formulate our initial estimate.
\begin{lemma}
\label{lem:starting_low_bound}
There is $\gamma>0$ (which do not depend on $k_0$) such that, with probability $1,$
\(
\inf_{n} M_{k_0}(n)/n^{\gamma} > 0.
\)
\end{lemma}
\begin{proof}
For fixed $n_0\in\mathbb{N}$, define 
$$C_{n+1}=\frac{4n}{4n-(c-\delta)^{d-1}}C_n =\left(1+\frac{(c-\delta)^{d-1}}{4n-(c-\delta)^{d-1}}\right)C_n,\,n\geq n_0,$$
with $C_{n_0} = 1$ and $0<\delta<c$. By Lemma~\ref{lem:numbers}, we have that $C_n/n^{(c-\delta)^{d-1}/4}$ converges to a positive limit.

Introduce events 
$$Q_{k_0}(n_0)=\{N_{k_0-1}<n_0,\hat{c}_{k_0}(n)>c-\delta\;\forall n>n_0\}.$$
By the induction hypothesis ($N_{k_0-1}<\infty$ almost surely and $\hat{c}_{k_0}(n)\to c$ in probability), $\mathbb{P}(Q_{k_0}(n_0))\to 1$ when $n_0\to\infty$.
Introduce Markov moments
$$\lambda_{k_0}(n_0)=
\inf\{n> n_0:L_{k}(n)>1\;\text{for some}\;k<k_0\;\text{or}\;\hat{c}_{k_0}(n)<c-\delta\}.$$
Note that $\lambda_{k_0}(n_0)=\infty$ on $Q_{k_0}(n_0)$. Put $A_{k_0}(n)=C_{n}/M_{k_0}(n)$. We will prove that $A_{k_0}(n\wedge\lambda_{k_0}(n_0))$ (where $x\wedge y=\min(x,y)$)  is a supermartingale for $n>n_0$. Hence by Doob's theorem (Corollary 3, p. 509 of \cite{Shir96}) it converges almost surely to some finite limit. Therefore, there is a random variable $B_{k_0,n_0},$ which is positive on $Q_k(n_0)$, such that $\frac{M_{k_0}(n)}{n^{\gamma}}\geq B_{k_0,n_0}$ almost surely for $\gamma=(c-\delta)^{d-1}/4$ and $n\geq n_0$. Consequently, we have
$$\mathbb{P}\left(\inf_{n\in\mathbb{N}}\frac{M_{k_0}(n)}{n^{\gamma}}>0\right) \geq\mathbb{P}\left(\inf_{n\in\mathbb{N}}\frac{M_{k_0}(n)}{n^{\gamma}}>0,Q_{k_0}(n_0)\right) 
=\mathbb{P}(Q_{k_0}(n_0))\rightarrow 1$$
Now prove that $A_{k_0}(n\wedge\lambda_{k_0}(n_0))$ is a supermartingale, which concludes our proof.
 
Recall that if $M_{k_0-1}(n)>M_{k_0}(n)$ (holds if $L_{k_0-1}=1$, in particualr for $n_0<n<\lambda_{k_0}(n_0)$) then $p_{n,k_0}$ equals to the probability to increase $k_0$-th maximal degree at the $n$-{th} step conditional on $\mathcal{F}_{n}$.
Note that
\begin{align*}
p_{n,k_0} &=\hat{c}_{k_0}^d(n) - \left(\hat{c}_{k_0}(n)-\frac{M_{k_0}(n)L_{k_0}(n)}{2n}\right)^{d} 
\geq \hat{c}_{k_0}^d(n)\left(1-\left(1-\frac{M_{k_0}(n)}{2n\hat{c}_{k_0}(n)}\right)^{d}\right)  \\ 
&\geq \hat{c}_{k_0}^d(n)\left(1-\left(1-\frac{M_{k_0}(n)}{2n\hat{c}_{k_0}(n)}\right)^{2}\right)
=\hat{c}_{k_0}^d(n)\left(\frac{M_{k_0}(n)}{n\hat{c}_{k_0}(n)}-\frac{(M_{k_0}(n))^2}{(2n\hat{c}_{k_0}(n))^2}\right)   \\
&= \hat{c}_{k_0}^{d-1}(n)\frac{M_{k_0}(n)}{n}\frac{4n\hat{c}_{k_0}(n)-M_{k_0}(n)}{4n\hat{c}_{k_0}(n)}\geq \hat{c}_{k_0}^{d-1}(n)\frac{2M_{k_0}(n)}{4n}=\hat{c}_{k_0}^{d-1}(n)\frac{M_{k_0}(n)}{2n}.
\end{align*}
By definition of $p_{n,k_0}$, for $1/M_{k_0}(n+1)$ we get
$$\E\left(1/M_{k_0}(n+1)|\mathcal{F}_n \right)=$$
$$\E\left(\left.\frac{\1\{M_{k_0}(n+1)=M_{k_0}(n)+1\}}{M_{k_0}(n)+1} +\frac{\1\{M_{k_0}(n+1)=M_{k_0}(n)\}}{M_{k_0}(n)}\right| \mathcal{F}_n\right)=$$ $$\left(\frac{p_{n,k_0}\1\{M_{k_0-1}(n)>M_{k_0}(n)\}}{M_{k_0}(n)+1}+\frac{1-p_{n,k_0}\1\{M_{k_0-1}(n)>M_{k_0}(n)\}}{M_{k_0}(n)}\right)=$$
$$\left(\frac{M_{k_0}(n)+1-p_{n,k_0}\1\{M_{k_0-1}(n)>M_{k_0}(n)\}}{(M_{k_0}(n)+1)M_{k_0}(n)}\right).$$ 
Therefore, if $n_0<n<\lambda_{k_0}(n_0)$, then 
$$\E\left(1/M_{k_0}(n+1)|\mathcal{F}_n \right)=\frac{M_{k_0}(n)+1-p_{n,k_0}}{M_{k_0}(n)(M_{k_0}(n)+1)}
=\frac{1}{M_{k_0}(n)}\left(1-\frac{p_{n,k_0}}{M_{k_0}(n)+1}\right)\leq$$ 
$$\frac{1}{M_{k_0}(n)}\left(1-\frac{p_{n,k_0}}{2M_{k_0}(n)}\right)
\leq\frac{1}{M_{k_0}(n)}\left(1-\frac{\hat{c}_{k_0}^{d-1}(n)}{4n}\right)$$
$$\leq\frac{1}{M_{k_0}(n)}\left(1-\frac{(c-\delta)^{d-1}}{4n}\right).$$
which concludes the proof.
\end{proof}

\section{Persistent hub}
\label{sec:hub}
We assume that Theorem~\ref{thm:max_degree} and Proposition~\ref{prop:persistent_hub} holds for $k<k_0$. In this section, we prove Proposition~\ref{prop:persistent_hub} for $k=k_0$ under this assumption.
Our method of the proof bases on the comparison of our model with the standart preferential attachment model, and we use the technique of~\cite{Galashin} developed for the last one.  
We divide the proof of Proposition~\ref{prop:persistent_hub} into two parts. First, we prove that degrees of only finite number of vertices could at some time become $k$-th maximal. Second, we prove that two vertices could have a $k$-th highest degree at the same time only for finite number of time moments. 

Let us introduce some notations:
$$\chi_k(n)=\min\{i\geq n:\deg v_n(i)=M_{k}(i)\},$$
$$U_k=\sum_{n=1}^{\infty}\1\{\chi_{k}(n)<\infty\},$$
$$\psi_{i,j}(n)=\min_{l\geq n}\{\deg v_{i}(l)=\deg v_{j}(l)\}.$$
Here $U_k$ is the number of vertices (of $V$) whose degrees were $k$-th maximal at some moments, $\chi_k(n)$ is the moment it happens for the vertex $v_n$. 

\begin{lemma}
\label{lem:change_of_leader}
$U_k$ is finite almost surely.
\end{lemma}

To prove the lemma, we first need a result (which is stated below) from~\cite{MP14} on a random walk that describes the evolution of degrees of two vertices in the preferential attachment model without choices.

Let $T_n=T_n(n_0,A_{n_0},B_{n_0}) = (A_n, B_n)$ for $n \geq n_0$ be random walks on $\Z^2$ started from some point $(A_{n_0},B_{n_0})$ that at time $n$ move one step right or one step up with the conditional probabilities $\frac{A_n}{A_n+B_n}$ and $\frac{B_n}{A_n+B_n}$ respectively. Also, indroduce the stoping times $\pi(i,j)=\min\{n\geq n_0: A_{n}=B_{n}|A_{n_0}=i, B_{n_0}=j\}$ and the function $q(i,j)=\mathbb{P}(\pi(i,j)<\infty)$. Although, the arguments of $q$ and $\pi$ are integers, sometimes in estimates we will write noninterges in arguments meaning the value of the floor function of it.

Lemma 4.2 from~\cite{MP14} stated that
\begin{lemma}
\label{lem:walk_domination}
The following inequality holds for any positive integers $i$ and $j$
$$\mathbb{P}(\psi_{i,j}(n)<\infty|\mathcal{F}_{n})\leq q(\deg v_{i}(n),\deg v_{j}(n)).$$
\end{lemma}

Let us prove Lemma \ref{lem:change_of_leader}.
\begin{proof}
By Lemma~\ref{lem:starting_low_bound}, we get 
\(
M_k(n)\geq M n^{\gamma}
\)
for some random $M>0$ almost surely. Hence, at time $n$ there are at least $k$ vertices $v_{i_1},...,v_{i_k}$ with degrees not less then $Mn^{\gamma}$ with probability 1. A degree of the vertex $v_{n+1}$ could become $k$-th maximal only if at some moment $\tilde{n}>n$ its degree becomes higher than at least one of the degrees $\deg v_{i_1}(\tilde{n}),...,\deg v_{i_k}(\tilde{n})$.
Due to Lemma~\ref{lem:walk_domination} (as in~\cite{MP14}), we could construct $k$ versions $\pi_{l}(i,j)$, $1\leq l\leq k$, of $\pi(i,j)$, such that $$\1\{\chi_k(n+1)<\infty\}\leq\sum_{l=1}^{k}\1\{\psi_{i_l,n+1}(n)<\infty\}\leq$$
$$\sum_{l=1}^{k}\1\{\pi_{l}(\deg v_{i_l}(n),1)<\infty\}\leq \sum_{l=1}^{k}\1\{\pi_{l}(Mn^{\gamma},1)<\infty\}\quad a.s.$$
Fix $C>0$. Then
$$U_{k}\1\{M>C\}=\sum_{n=1}^{\infty}\1\{\chi_k(n)<\infty\}\1\{M>C\} <$$ $$\sum_{n=1}^{\infty}\sum_{l=1}^{k}\1\{\pi_{l}(Mn^{\gamma},1)<\infty\}\1\{M>C\}\leq
\sum_{n=1}^{\infty}\sum_{l=1}^{k}\1\{\pi_{l}(Cn^{\gamma},1)<\infty\}\1\{M>C\}\leq$$
$$\sum_{n=1}^{\infty}\sum_{l=1}^{k}\1\{\pi_{l}(Cn^{\gamma},1)<\infty\}.$$
 
Corollary 15 of~\cite{Galashin} gives us the following estimate:
$$q(i,1)\leq\frac{Q(i)}{2^i}\;\text{for any integer}\;i$$
for some polynomial function $Q(x)$. Therefore, the expectations
$$\mathbb{E}\1\{\pi_{l}(Cn^{\gamma},1)<\infty\}=q(Cn^{\gamma},1)\leq\frac{Q(Cn^{\gamma})}{2^{Cn^{\gamma}}}$$
forms a convergent series, and the last sum is finite almost surely by Borel-Cantelli Lemma. Since $M>0$ with probability $1$,
$$\mathbb{P}(U_{k}<\infty)=\mathbb{P}(\{U_{k}<\infty\}\bigcup\bigcup_{n\in\mathbb{N}}\{M>1/n\})=1.$$
\end{proof}

Now let $J_k$ denote the set of vertices whose degrees become $k$-th maximal at some moment. According to Lemma~\ref{lem:change_of_leader}, $J_k$ is finite almost surely. Introduce random moments
$$\zeta_l(v_i,v_j)=\inf\{n>\zeta_{l-1}(v_i,v_j):$$
$$\deg v_i(n-1)\neq\deg v_j(n-1)\,\text{and}\,\deg v_i(n)=\deg v_j(n) \},\;\zeta_0(v_i,v_j)=0,$$
$$N(v_i,v_j)=\sup\{l:\zeta_{l}(v_i,v_j)<\infty\},$$
$$\xi_k=\sup\{\zeta_{N(v_i,v_j)}(v_i,v_j)|v_i\in J_k,v_j\in J_k\}.$$
Note that almost sure finitness of $\xi_k$ implies Proposition~\ref{prop:persistent_hub} cause any vertex that become $k$-th maximal at any time is in $J_k$, and an order of degrees of vertices from $J_k$ does not change after the moment $\xi_k$. Thus, to complete the proof of Proposition~\ref{prop:persistent_hub} we need the following lemma:
\begin{lemma}
\label{lem:change_of_leadership}
$\xi_k$ is finite almost surely.
\end{lemma}
\begin{proof}
Since $J_k$ is finite almost surely, it is enough to prove that for any $v_i,v_j\in V$ $N(v_i,v_j)$ is finite almost surely. To do so we will use the random walk $T_n$ with $n_0=\max\{i+1,j+1\}$, $A_{n_0}=\deg v_i(n_0)$, $B_{n_0}=\deg v_j(n_0)$. Let $R(n_0,i,j)$ be the number of times $n>n_0$ such that $A_n=B_n$, and let $n_0\leq\rho_1(i,j)<\rho_2(i,j)<...$ be moments when either $\deg v_i$ or $\deg v_j$ is changed. Then due to the coupling used in the proof of Lemma 4.2 from~\cite{MP14} there is version of $T$, such that $\min\{\deg v_i(\rho_n),\deg v_j(\rho_n)\}$ is dominated by $\min\{A_{n},B_{n}\}$ for $n\geq n_0$, which implies $N(v_i,v_j)\leq R$ (since $A_n+B_n=\deg v_i(\rho_n)+\deg v_j(\rho_n)$). 

It is a standard fact about P\'olya urn model that if $T_n=(A_n,B_n)$ starts from a point $(a,b)$, then the fraction $A_{n}/(A_{n}+B_{n})$ tends in law to a random variable $H(a,b)$ as $n$ tends to infinity, where $H(a,b)$ has beta probability distribution:
$$H(a,b)\sim\Beta(a,b).$$
(See, e.g., Theorem 3.2 in \cite{M09} or Section 4.2 in \cite{JK77}). Thus, the limit of $A_n/(A_n + B_n)$ exists almost surely, and it takes the value $1/2$ with probability $0$ for any starting point of the process $T$.  Hence, this fraction can be equal to $1/2$ only finitely many times almost surely, and so $R$ is finite almost surely, which completes the proof.
\end{proof}

\section{Final result}
\label{sec:finald}

Fix $0<\delta<2/d-c^{d-1}$ (by Lemma~\ref{lem:c_est}, $2/d>c^{d-1}$). For any fixed $n_0>0$, we introduce the events $$D_{k_0}(n_0,\delta)=\{L_l(n)=1,\,c-\delta<\hat{c}_{k_0}(n)<c+\delta,\,M_{k_0}(n)>n^{\gamma}/n_0 \;\forall n\geq n_0\;\forall k\leq k_0\},$$
and the Markov moment
$$\eta_{k_0}(n_0,\delta)=\inf\{n \geq n_0:L_k(n) > 1\,\text{for some}\, k\leq k_0,\, \text{or}\,$$ $$\hat{c}_{k_0}(n)>c+\delta,\,\text{or}\,\hat{c}_{k_0}(n)<c-\delta,\,\text{or}\,M_{k_0}(n)\leq n^{\gamma}/n_0\}.$$
Note that by the induction assumption for $k<k_0$, Proposition~\ref{prop:persistent_hub} and Lemma~\ref{lem:starting_low_bound} (both for $k=k_0$) we have that $$\mathbb{P}(D_{k_0}(n_0,\delta))=\mathbb{P}(\eta_{k_0}(n_0,\delta)=\infty)\to 1\, \text{as}\, n_0\to\infty.$$
Now, let prove Theorem~\ref{thm:max_degree} for $k=k_0$
\begin{lemma}
\label{lem:final_est}
With probability $1,$
\(
M_k(n)/n \to 0.
\)
\end{lemma}
\begin{proof}


Recall that if $M_{k-1}(n)>M_k(n)$ (in particular, on $D_k(n0,\delta)$) then $p_{n,k}$ equals to the conditional probability to increase $M_k(n)$ conditional on $\filt_n.$ Note that for $n$ such that $n_0 \leq n \leq \eta_{k,C},$
\[
p_{n,k}
=\hat{c}_k^d(n)-\left(\hat{c}_k(n)-\frac{M_k(n)}{2n}\right)^{d}
=\frac{M_k(n)}{2n}\left(\sum_{i=0}^{d-1}\hat{c}_k^{d-i-1}(n)\left(\hat{c}_k(n)-\frac{M_k(n)}{2n}\right)^{i}\right).
\]
Hence, $\frac{p_{n,k}}{M_k(n)}=\frac{1}{n}f(\frac{M_k(n)}{n},\hat{c}_k(n))$.
From Lemmas~\ref{lem:f_est},~\ref{lem:c_est}, it follows that for any small enought $\delta>0$ there is $\beta>0$ so that $f(x,y)<1-\beta$ if $y<c+\delta$ for $0\leq x\leq 2y$.

Consider the expectation:
\[
\Exp\left(\left.\frac{M_k(n+1)}{M_k(n)} \right| \filt_{n}\right)
=\frac{p_{n,k}(M_k(n)+1)}{M_k(n)}+1-p_{n,k}
=1+\frac{p_{n,k}}{M_k(n)}.
\]

Therefore, for small enough $\delta>0$ there is $\beta>0$ such that $\Exp(M_k(n+1)|\filt_{n})<(1+(1-\beta)/n)M_k(n)$ for $n_0\leq n<\eta_{k}(n_0,\delta)$. Set $A_{k}(n) = M_k(n)/C_{n,k},$ where $C_{n+1,k}=(1+(1-\beta)/n)C_{n,k},$ $n\geq n_0$, $C_{n_0,k}=1$. We have that
$$\mathbb{E}\left(\left.\frac{A_k(n+1\wedge\eta_k(n_0,\delta))}{A_k(n\wedge\eta_k(n_0,\delta))}\right|\mathcal{F}_{n}\right)=$$
$$\mathbb{E}\left(\left.\frac{A_k(n+1)}{A_k(n)}\1\{n+1\leq\eta_k(n_0,\delta)\}+\1\{n+1>\eta_k(n_0,\delta)\}\right|\mathcal{F}_n\right)=$$
$$\1\{\eta_k(n_0,\delta)>n\}\mathbb{E}\left(\left.\frac{A_k(n+1)}{A_k(n)}\right|\mathcal{F}_n\right)+ \1\{\eta_k(n_0,\delta)\leq n\}=$$
$$\1\{\eta_k(n_0,\delta)>n\}\frac{C_{n,k}}{C_{n+1,k}}\mathbb{E}\left(\left.\frac{M_k(n+1)}{M_k(n)}\right|\mathcal{F}_n\right) +\1\{\eta_k(n_0,\delta)\leq n\}\leq$$
$$\1\{\eta_k(n_0,\delta)>n\}\frac{1+(1-\beta)/n}{1+(1-\beta)/n}+\1\{\eta_k(n_0,\delta)\leq n\}=1.$$
Thus, $A_k(n\wedge\eta_k(n_0,\delta))$ is a supermartingale. By Lemma~\ref{lem:numbers}, we have that $C_{n,k}n^{-1+\beta}$ converges to a positive limit. Therefore, by Doob's theorem we have that $A_k(n\wedge\eta_k(n_0,\delta))$ tends to a finite limit with probability 1, and, in particular, there is a random constant $B_k > 0$ so that $M_k(n\wedge\eta_k(n_0,\delta)) \leq B_k n^{1-\beta}$ almost surely.
Thus, $M_k(n\wedge\eta_k(n_0,\delta))/n \to 0$ almost surely as $n\to\infty$, and, since $\mathbb{P}(\eta_k(n_0,\delta)=\infty)\to 1$ as $n_0\to\infty$, $M_k(n)/n\to 0$ almost surely as $n\to\infty$.
\end{proof}

Now, concider the expectation for $n_0\leq n\leq\eta_{k}(n_0,\delta)$ and some $0<\alpha<1$
$$\mathbb{E}\left(\left.\frac{M_k(n+1)/(n+1)^{\alpha}}{M_k(n)/n^{\alpha}}\right|\mathcal{F}_n\right)= \frac{M_k(n)+\hat{c}_k^d(n)-\left(\hat{c}_k(n)-\frac{M_k(n)}{2n}\right)^{d}}{M_k(n)}\frac{n^{\alpha}}{(n+1)^{\alpha}}=$$
$$\left(1+\frac{1}{2n}\left(d\hat{c}_k^{d-1}(n)+\sum_{i=1}^{d-1}(-1)^i {i\choose d}\hat{c}_k^{d-1-i}(n)\left(\frac{M_k(n)}{2n}\right)^{i}\right)\right)\frac{1}{(1+1/n)^{\alpha}}.$$
By the induction assumption and Lemma~\ref{lem:final_est},
$$\Delta_k(n)=d\hat{c}_k^{d-1}(n)+\sum_{i=1}^{d-1}(-1)^i {i\choose d}\hat{c}_k^{d-1-i}(n)\left(\frac{M_k(n)}{2n}\right)^{i}\to c^{d-1}d \text{ a.s. as } n\to\infty.$$
In particular, for any $\epsilon>0$ 
$$\mathbb{P}(\Delta_k(n)<c^{d-1}d+\epsilon\,\text{for}\,n>n_0)\to 1 \text{ and}$$
$$\mathbb{P}(\Delta_k(n)>c^{d-1}d-\epsilon\,\text{for}\,n>n_0)\to 1\, \text{as}\,n_0\to\infty.$$
Therefore,
$$\mathbb{P}\left(\mathbb{E}\left(\left.\frac{M_{k}(n+1)/(n+1)^{c^{d-1}d/2-\epsilon}}{M_{k}(n)/n^{c^{d-1}d/2-\epsilon}}\right|\mathcal{F}_n\right)>1,\, \text{for}\, n>n_0\right)\geq$$
$$\mathbb{P}\Bigg(\mathbb{E}\left(\left.\frac{M_{k}(n+1)/(n+1)^{c^{d-1}d/2-\epsilon}}{M_{k}(n)/n^{c^{d-1}d/2-\epsilon}}\right|\mathcal{F}_n\right)>1,\,$$ $$\Delta_k(n)>c^{d-1}d-\epsilon\, \text{for}\, n>n_0, \eta_{k}(n_0,\delta)=\infty\Bigg)\geq$$
$$\mathbb{P}\Bigg(\frac{1+(c^{d-1}d/2-\epsilon/2)/n}{(1+1/n)^{c^{d-1}d/2-\epsilon}}>1,$$
$$\Delta_k(n)>c^{d-1}d-\epsilon,\, \text{for}\, n>n_0, \eta_{k}(n_0,\delta)=\infty\Bigg)\to 1\,\text{as}\,n_0\to\infty\, \text{and}$$
$$\mathbb{P}\left(\mathbb{E}\left(\left.\frac{M_{k}(n+1)/(n+1)^{c^{d-1}d/2+\epsilon}}{M_{k}(n)/n^{c^{d-1}d/2+\epsilon}}\right|\mathcal{F}_n\right)<1,\, \text{for}\, n>n_0\right)\geq$$
$$\mathbb{P}\Bigg(\mathbb{E}\left(\left.\frac{M_{k}(n+1)/(n+1)^{c^{d-1}d/2+\epsilon}}{M_{k}(n)/n^{c^{d-1}d/2+\epsilon}}\right|\mathcal{F}_n\right)<1,\,$$ $$\Delta_k(n)<c^{d-1}d+\epsilon\, \text{for}\, n>n_0, \eta_{k}(n_0,\delta)=\infty\Bigg)\geq$$
$$\mathbb{P}\Bigg(\frac{1+(c^{d-1}d/2+\epsilon/2)/n}{(1+1/n)^{c^{d-1}d/2+\epsilon}}<1,\,$$ $$\Delta_k(n)<c^{d-1}d+\epsilon\, \text{for}\, n>n_0, \eta_{k}(n_0,\delta)=\infty\Bigg)\to 1$$
as $n_0\to\infty$. Introduce Markov moments 
$$\nu_{k,n_0,\epsilon}=\inf\left\{\left.n>n_0: \mathbb{E}\left(\frac{M_{k}(n+1)/(n+1)^{c^{d-1}d/2-\epsilon}}{M_{k}(n)/n^{c^{d-1}d/2-\epsilon}}\right|\mathcal{F}_n\right)\leq 1, \text{or}\right.$$
$$\left.\mathbb{E}\left(\left.\frac{M_{k}(n+1)/(n+1)^{c^{d-1}d/2+\epsilon}}{M_{k}(n)/n^{c^{d-1}d/2+\epsilon}}\right|\mathcal{F}_n\right)\geq 1\right\}.$$
Note that $\mathbb{P}(\nu_{k,n_0,\epsilon}=\infty)\to 1$ as $n_0\to\infty$. Let $$A_{k}(n)=\frac{M_{k}(n)}{n^{c^{d-1}d/2+\epsilon/2}} \text{ and } B_{k}(n)=\frac{n^{c^{d-1}d/2-\epsilon/2}}{M_{k}(n)}.$$
Then $A_k(n\wedge\nu_{k,n_0,\epsilon})$ and $B_k(n\wedge\nu_{k,n_0,\epsilon})$ are supermartingales, and from Doob's theorem, $$\frac{M_k(n)}{n^{c^{d-1}d/2-\epsilon}}\to \infty \text{ and } \frac{M_k(n)}{n^{c^{d-1}d/2+\epsilon}}\to 0 \text{ almost surely},$$
which imply our theorem.
\section*{Acknowledgements.}
The author is grateful to Professor Itai Benjamini for proposing the model and to Maksim Zhukovskii for helpful discussions.
\bibliographystyle{alpha}
\bibliography{High degree vertices}

\end{document}